\documentclass{amsart}
\usepackage{graphicx}
\usepackage{amscd}
\usepackage{amsmath}
\usepackage{amsfonts}
\usepackage{amssymb}

\theoremstyle{plain}
\newtheorem{theorem}{Theorem}[]

\newtheorem{lemma}[theorem]{Lemma}

\theoremstyle{definition}

\theoremstyle{remark}

\numberwithin{equation}{section}

\email{}
\email{shafiq\_ur\_rahman2@yahoo.com, shafiq@ciit-attock.edu.pk}

\begin{document}
\title{{A note on a characterization theorem for a certain class of domains}}    
\author{Shafiq ur Rehman}

\address{(Rehman)  COMSATS Institute of Information Technology, Attock, Pakistan and \indent ASSMS, GC University Lahore, Pakistan.}

\thanks{2010 Mathematics Subject Classification: Primary 13A15, Secondary 13F05.}
\thanks{Key words and phrases: Invertible ideal, Pr\"{u}fer domain, h-local domain.}

\begin{abstract}
\noindent
We have introduced and studied in \cite{DumII} the class of {\em Globalized multiplicatively pinched-Dedekind domains $($GMPD domains$)$}. This class of  domains could be characterized by a certain  factorization property of the non-invertible ideals, (see \cite[Theorem 4]{DumII}). In this note a simplification of the characterization theorem  \cite[Theorem 4]{DumII}  is provided in more general form.
\end{abstract}

\maketitle

Let $D$ be an integral domain. By an {\em MNI ideal} of $D$ we mean
an ideal of $D$ which is  maximal among the nonzero noninvertible ideals of $D$. By \cite[Exercise 36, page 44]{Kap}, every MNI ideal is a prime ideal. Moreover, using standard  Zorn's Lemma arguments, one can show that every nonzero non-invertible ideal is contained in some MNI ideal. $D$ is said to be {\em h-local} provided every nonzero ideal of $D$ is contained in at most finitely many maximal ideals of $D$ and each nonzero prime ideal of $D$ is contained in a unique maximal ideal of $D$.
$D$ is called a {\em pseudo-valuation domain (PVD)} if $D$ is quasi-local with maximal ideal $M$ and $M:M$ is a valuation domain with maximal ideal $M$, cf.  \cite{HH} and \cite[Proposition 2.5]{AD}. A {\em two-generated} domain is a domain whose ideals are two-generated.
Let $D$ be a quasi-local domain with maximal ideal $M$. By \cite[Theorems 2.7 and 3.5]{HH}, $D$ is a two-generated PVD  if and only if   $D$ is a field,   a DVR, or   a Noetherian domain such that its integral closure $D'$ is a DVR with maximal ideal $M$ and  $D'/M$ is a quadratic field extension of $D/M$.\\

In \cite{DumII}, we introduced and study the class of  {\em Globalized multiplicatively pinched-Dedekind domains $($GMPD domains$)$}. A domain $D$ is called a {\em globalized multiplicatively pinched-Dedekind domain  $($GMPD domain$)$} if $D$ is h-local and for each maximal ideal $M$, $D_M$ is  a  two-generated PVD, or  a valuation domain with value group    $\mathbb{Z}\times \mathbb{Z}$ or $\mathbb{R}$, cf. \cite[Definition 2]{DumII}. A Dedekind domain is a GMPD domain and the integrally closed Noetherian GMPD domain are exactly the Dedekind domains. This class of domains could  be characterized by a certain  factorization property of the non-invertible ideals. An h-local domain $D$ is a GMPD domain if and only if
every two MNI ideals are comaximal and every  nonzero non-invertible ideal $I$ of $D$ can be written as $I=JP_1\cdots P_k$ for some invertible ideal $J$ and  distinct MNI ideals $P_1,...,P_k$  uniquely determined by $I$ (\cite[Theorem 4, Remark 5]{DumII}). In this note a more general simplification of this characterization theorem in provided (Theorem \ref{3}).\\

Throughout this note all rings are (commutative unitary) integral domains. For a domain $D$, $D'$ (resp. $\bar{D}$)  denotes the integral closure (resp. complete integral closure) of $D$.  Any unexplained material is standard like in \cite{G} or \cite{Kap}.

\begin{lemma}\label{1}
Every two distinct MNI ideals of a domain $D$ are comaximal.
\end{lemma}
\begin{proof}
 Deny. Let $P_1 \neq P_2$ be the MNI ideals of $D$ both contained in the maximal ideal $M$. Then $M$ is invertible and so $P_i \subsetneq M$ implies that $P_i \subsetneq \cap_{n\geq 1} M^n=Q$. The ideal $Q$ is invertible and prime. Indeed, if $ab \in Q$ with both $a,b \not \in Q$, then there exist integers $k, l$ and the ideals $U,V$  such that $(a)=M^kU$ with $U \nsubseteq M$ and $(b)=M^lV$ with $V \nsubseteq M$. Since $ab \in M^{k+l+1}$, so $(ab)=M^{k+l+1}N$ for some ideal $N$. Combining, we get that $M^{k+l}UV=M^{k+l+1}N$. This implies that $UV=MN$ which is not possible because $U,V$ are not contained in $M$. Hence $Q$ is prime. As $P_i \subsetneq Q$, so $Q$ is invertible. Since any two invertible prime ideals are not comparable, so $Q=M$. This implies that $M=M^2$ and hence $M=D$, a contradiction.
\end{proof}

Recall \cite[Section 5.1]{FHL} that a domain  $D$ has pseudo-Dedekind factorization if for
each nonzero non-invertible ideal $I$, there is an invertible ideal $B$ (which might
be $D$) and finitely many pairwise comaximal primes $P_1,P_2,...,P_n$ such that
$I=BP_1P_2 \cdots P_n$ (the requirement that $n > 0$ "comes for free").

\begin{theorem}\label{2}
Let $D$ be a domain such that every nonzero non-invertible ideal $I$ of $D$ can be written as $I=JP_1\cdots P_k$
for some invertible ideal $J$ and  distinct MNI ideals $P_1,...,P_k$. Then $D$ is h-local.
\end{theorem}
\begin{proof}
By \cite[Exercise 36, page 44]{Kap}, every MNI ideal is a prime ideal and by Lemma \ref{1}, $P_1,...,P_k$ are pairwise  comaximal. Hence $D$ has pseudo-Dedekind factorization, cf. \cite[Section 5.1]{FHL}. Now Apply \cite[Corollary 5.2.14]{FHL}.
\end{proof}

\begin{theorem}\label{3}
A domain $D$ is a GMPD domain if and only if
every  nonzero non-invertible ideal $I$ of $D$ can be written as $I=JP_1\cdots P_k$
for some invertible ideal $J$ and  distinct MNI ideals $P_1,...,P_k$.
\end{theorem}
\begin{proof}
Apply Theorem \ref{2} and  \cite[Theorem 4]{DumII}.
\end{proof}

\end{document}